\PassOptionsToPackage{hyphens}{url}
\documentclass{amsart}
\usepackage{lmodern}
\usepackage[capitalize]{cleveref}
\usepackage{amssymb,amsmath,amsthm}
\usepackage{ifxetex,ifluatex}
\usepackage{setspace}
\usepackage{tikz-cd}
\usepackage[T1]{fontenc}
\usepackage[utf8]{inputenc}
\usepackage{textcomp} 

\usepackage{xcolor}
\providecommand{\tightlist}{%
  \setlength{\itemsep}{0pt}\setlength{\parskip}{0pt}}
\usepackage{amsthm}
\usepackage{thmtools}
\usepackage{enumitem}
\usepackage[backend=biber,url=false,maxbibnames=9]{biblatex}

\newcommand{\elbe}{\approxeq}
\newcommand{\biemb}{\approx}

\newcommand{\A}{\mathcal{A}}
\newcommand{\C}{\mathcal{C}}
\newcommand{\B}{\mathcal{B}}
\renewcommand{\S}{\mathcal{S}}
\newcommand{\ol}[1]{\overline{#1}}

\newcommand{\mf}[1]{\mathfrak{#1}}

\newcommand{\mc}[1]{\mathcal{#1}}
\newcommand{\ra}{\rightarrow}
\newcommand{\Ra}{\Rightarrow}
\newcommand{\La}{\Leftarrow}
\newcommand{\LR}{\Leftrightarrow}
\newcommand{\edge}{\rightarrowtail}
\renewcommand{\phi}{\varphi}

\newcommand{\I}{\MakeUppercase{\romannumeral 1}}
\newcommand{\II}{\MakeUppercase{\romannumeral 2}}

\DeclareMathOperator{\emb}{\hookrightarrow}
\DeclareMathOperator{\eemb}{\preccurlyeq}
\DeclareMathOperator{\ar}{\leadsto}
\declaretheorem[name={Theorem}]{theorem}
\declaretheorem[name={Proposition}, sibling=theorem]{proposition}
\declaretheorem[name={Lemma}, sibling=theorem]{lemma}

\declaretheorem[name={Definition}, style=definition,sibling=theorem]{definition}
\declaretheorem[name={Corollary}, sibling=theorem]{corollary}
\declaretheorem[name={Question}]{question}
\ifluatex
  \usepackage{selnolig}  
\fi

\newcommand\addrand{{\normalfont and}}   

\newcommand{\twoaddress}[4][]{\g@addto@macro\addresses{%
\address{#1}{#2 \protect\advance\parindent by -#4em\protect\\\addrand %
\protect\\ #3}}}

\title{Degree spectra of analytic complete equivalence relations}
\author{Dino Rossegger}
\date{\today}
\address{Department of Mathematics, University of California, Berkeley
\addrand\ Institute of Discrete Mathematics and Geometry, Technische Universit\"at Wien}

\email{dino@math.berkeley.edu}

\thanks{The author thanks David Marker for providing references on minimality.
The results in this paper were obtained while the author was a postdoctoral
fellow at the Department of Pure Mathematics of the University of Waterloo. The
author was partially supported by a Marie Sk\l{}odowska-Curie Action under
grant agreement No. 101026834 — ACOSE}

\begin{document}
\maketitle
\begin{abstract}
  We study the bi-embeddability and elementary bi-embeddability relation on
  graphs under Borel reducibility and investigate the degree spectra realized
  by these relations. We first give a Borel reduction from embeddability on
  graphs to elementary embeddability on graphs. As a consequence we obtain that
  elementary bi-embeddability on graphs is a $\pmb \Sigma^1_1$ complete
  equivalence relation. We then investigate the algorithmic properties of
  this reduction. We obtain that elementary bi-embeddability on the class of
  computable graphs is $\Sigma^1_1$ complete with respect to computable
  reducibility and show that the elementary bi-embeddability and
  bi-embeddability spectra realized by graphs are related. 
\end{abstract}
\section{Introduction}\label{introduction}
Equivalence relations on countable structures are among the most heavily
studied objects in descriptive set theory and computability theory. In
descriptive set theory, starting with the work of Friedman and
Stanley~\cite{friedman1989}, the complexity of equivalence relations on spaces
of structures under Borel reducibility has seen much interest by experts,
culminating in results by Louveau and Rosendal~\cite{louveau2005}, who showed
that, among others, the bi-embeddability relation on graphs is $\pmb
\Sigma^1_1$-complete. Since then there has been a constant stream of work on
the complexity of the bi-embeddability relation, both on other classes of
structures, see for instance~\cite{calderoni2019}, and refinements of completeness
notions, e.g. in~\cite{camerlo2013}.

Equivalence relations are also one of the main objects of study in
computability theory. Here, the equivalence relations are usually on the set of natural
numbers and their complexity is established using computable
reducibility. Identifying a computable structure
with the index of the algorithm computing it, one can obtain completeness
results like the ones in descriptive set theory for equivalence relations on
computable structures~\cite{fokina2009,fokina2012}. One object of
study in computable structure theory which also takes non-computable structures
into account are degree spectra of structures, introduced by
Knight~\cite{knight1986}. The degree spectrum of
a given structure is the set of sets of
natural numbers Turing equivalent to one of its isomorphic copies. 
They provide a measure of the algorithmic complexity of countable structures.

Recently, researchers initiated the study of degree spectra with respect to
other model theoretic equivalence relations such as
bi-embeddability~\cite{fokina2019a}, elementary
bi-embeddability~\cite{rossegger2018}, elementary
equivalence~\cite{andrews2015,andrews2017,andrews2013a}, or $\Sigma_n$
equivalence~\cite{fokina2016}. One of the main goals in this line of research
is to distinguish these equivalence relations with respect to the degree
spectra they realize. While for elementary equivalence and $\Sigma_n$ equivalence
examples that separate them from each other and from isomorphism and elementary
bi-embeddability are known, so far all attempts to separate
isomorphism, bi-embeddability and elementary bi-embeddability have been
unsuccessful.

There seem to be various reasons for this. That we can separate elementary
equivalence and $\Sigma_n$ equivalence is the case because they have different
levels in the Borel hierarchy while isomorphism and bi-embeddability are not even
Borel. On the other hand bi-embeddability preserves very little structural
properties and it is thus difficult to construct interesting examples.
The aim of this article is to investigate the relationship between the degree
spectra realized by the bi-embeddability relation and by the elementary
bi-embeddability relation. First, we establish that elementary
bi-embeddability on graphs is $\pmb \Sigma^1_1$ complete with respect to Borel
reducibility. We then proceed to establish a relationship between the degree spectra realized
by the bi-embeddability and elementary bi-embeddability relation on graphs. Our
main results are as follows.
\begin{theorem}\label{thm:elbeborel}
  The elementary embeddability relation on graphs $\eemb_\mf G$ is a complete $\pmb
  \Sigma^1_1$ quasi-order. In particular,
  the elementary bi-embeddability relation on graphs $\elbe_\mf G$ is a complete $\pmb \Sigma^1_1$ equivalence relation.
\end{theorem}

As a corollary of \cref{thm:elbeborel} we obtain the corresponding result for
elementary bi-embeddability on computable structures.
\begin{theorem}\label{thm:elbecomp}
  The elementary embeddability relation on the class of computable graphs is
  a $\Sigma^1_1$ complete quasi-order with respect to computable reducibility.
  In particular, the elementary bi-embeddability relation on computable graphs is
  a $\Sigma^1_1$ complete equivalence relation with respect to computable reducibility.
\end{theorem}
The following result establishes a relationship between bi-embeddability
spectyra of graphs and elementary bi-embeddability spectra of graphs. 
\begin{theorem}\label{thm:dgsp}
  Let $\mc G$ be an automorphically non-trivial graph, then there is a graph $\hat{\mc G}$ such that
  \[ DgSp_\elbe(\hat{\mc{G}})=\{X: X'\in DgSp_\biemb({\mc G})\}.\]
\end{theorem}
Note that in \cref{thm:dgsp} we deal only with automorphically non-trivial
graphs. This might seem like a shortcoming, however automorphically trivial
structures are not interesting from a computability theoretic point of view. In
particular, every structure bi-embeddable with an automorphically trivial graph is
computable and thus both its bi-embeddability spectrum and elementary
bi-embeddability spectrum is the set of all computable sets.

The proofs of \cref{thm:elbeborel,thm:elbecomp} are the topic of
\cref{sec:analcomplete}. In \cref{sec:degreespectra} we build on these results to
prove \cref{thm:dgsp}. In \cref{sec:background} we give the necessary
background and definitions.
\section{Background}\label{sec:background}
Our definitions follow for the most part~\cite{gao2008}
and~\cite{montalban2021a}. We assume that all structures have universe $\omega$ and are 
relational.
Let $\mc L$ be a relational language $(R_i)_{i\in\omega}$ where without loss of
generality $R_i$ has arity $i$. Then each element $\mc A$ of $Mod(\mc L)$ can be
viewed as an element of the product space
\[ X_\mc L=\prod_{i\in \omega} 2^{\omega^{i}}\]
and thus $Mod(\mc L)$ becomes a compact Polish space on which we can define the
Borel and projective hierarchy in the usual way.

Let $\mc A$ be an $\mc L$-structure and $(\phi_i^{at})_{i\in\omega}$ be
a computable enumeration of the atomic $\mc L$-sentences with variables in
$\{x_1,x_2,\dots\}$. The \emph{atomic diagram $D(\mc A)$} of $\mc A$ is the
element of Cantor space defined by
\[ D(\A)(i)=\begin{cases} 1 & \text{ if } \A\models \phi_i^{at}[x_j\ra
j : j\in\omega]\\
0 &\text{ otherwise.}\end{cases}\]
The Turing degree of a structure $\A$ is the degree of $D(\A)$. We will in
general not distinguish between a structure as an element of $Mod(\mc L)$ and its
atomic diagram and assume that what is meant is clear from the context.

Variations of the following definition were independently
suggested in~\cite{montalban2015a,fokina2016,yu2015}.
\begin{definition}
  Let $E$ be an equivalence relation on $Mod(\mc L)$ and $\A\in Mod(\mc L)$. Then the
  \emph{degree spectrum of $\A$ with respect to $E$}, or, short $E$-spectrum of
  $\A$, is the set
  \[ DgSp_E(\A)=\{ X: \exists \B \mathop{E} \A\ D(\B)\equiv_T X\}\]
\end{definition}
We write $\A\emb\B$ to say that $\A$ is \emph{embeddable} in $\B$, and
$\A\biemb \B$ to say that $\A$ is \emph{bi-embeddable} with $\B$, i.e.,
$\A\emb\B$ and $\B\emb \A$.
Further, we write $\A\eemb \B$ to say that $\A$ is \emph{elementary embeddable} in
$\B$ and $\A\elbe \B$ to say that $\A$ is \emph{elementary bi-embeddable} with $\B$,
i.e., $\A\eemb \B$ and $\B\eemb \A$.

\begin{definition}
  Let $R,S$ be binary relations on a set $X$. The relation $R$ is
  \emph{reducible} to $S$ if there is a function $f: X\ra X$ such that for all
  $x,y\in X$
  \[ x R y \LR f(x)Sf(y).\]
  Assume $X=Mod(\mc L)$. Then 
  \begin{enumerate}
    \item $R$ is \emph{Borel reducible} to $S$ if $f$ is Borel on $Mod(\mc
      L)\times Mod(\mc L) $,
    \item $R$ is \emph{computably reducible} to $S$ if there is a computable
      operator $\Phi$ such that for all $\A\in Mod(\mc L)$,
      $\Phi^{D(\A)}=D(f(\A))$.
  \end{enumerate}
  Assume $X$ is $\omega$ and that $(\A_i)_{i\in\omega}$ is a computable
  enumeration of all partial computable $\mc L$ structures. Then $R$ is \emph{computably
  reducible} to $S$ if $f$ is a computable function.
\end{definition}
We say that an equivalence relation (quasi-order) $R\in \Gamma$ is a $\Gamma$ complete equivalence relation (quasi-order) for a complexity class $\Gamma$
with respect to $x$-reducibility if all equivalence relations (quasi-orders) in $\Gamma$ are $x$-reducible to
$R$.

A standard reference on Borel reducibility is~\cite{gao2008}. 
Computable reducibility on the natural numbers can be seen as a natural
effectivization of Borel reducibility where one only considers computable
structures. Fokina and
Friedman~\cite{fokina2009} showed that bi-embeddability on trees and thus also
graphs is $\Sigma^1_1$
complete with respect to computable reducibility, and in~\cite{fokina2012} it
is shown that isomorphism on graphs is $\Sigma^1_1$ complete with respect to
computable reducibility. This contrasts with Borel reducibility; it is
well known that isomorphism on graphs is not $\pmb \Sigma^1_1$ complete.

\section{Elementary bi-embeddability is analytic complete}\label{sec:analcomplete}
In this section we prove \cref{thm:elbeborel,thm:elbecomp} and some lemmas needed for
\cref{thm:dgsp}. The section is structured as follows. In
\cref{sec:reduemb_geemb_c} we give a reduction from embeddability on the class
of graphs $\mf G$ to elementary embeddability on a Borel class $\mf C$ of
structures in an infinite relational language. In \cref{sec:graphscomplete} we
show that graphs are complete for elementary embeddability. That is, for
every Borel class, elementary embeddability on this class can be reduced to
elementary embeddability on graphs. \cref{thm:elbeborel} then follows by
composing the reductions given in
\cref{sec:reduemb_geemb_c,sec:graphscomplete}. For \cref{thm:elbecomp} we need
a few more observations made at the end of this section.
\subsection{The reduction from \(\emb_\mf G\) to
\(\eemb_{\mf C}\)}\label{sec:reduemb_geemb_c}

The main idea of the construction is that for any given
graph \(\mc G\) we replace the edge relation with structures having the
property that they are minimal under elementary embeddability.

\begin{definition}\label{def:minimality}
A structure \(\A\) is \emph{minimal} if it does not have proper
elementary substructures.
\end{definition}

Minimal structures were investigated by Fuhrken~\cite{fuhrken1966} who
showed that there is a theory with \(2^{\aleph_0}\) minimal models, and
Shelah~\cite{shelah1978} who showed that for every \(n\leq \aleph_0\),
there is a theory with \(n\) minimal models. Later,
Ikeda~\cite{ikeda1993} investigated minimal models of minimal theories.
Notice that a prime model is not necessarily minimal, as it might
contain elementary substructures isomorphic to itself.

Given a graph \(\mc G\), if \(x,y\in G\) and \(xEy\), then we
associate a copy of a structure \(\A\) with the pair \((x,y)\) and
otherwise we associate a copy \(\B\) with \((x,y)\). The
structures \(\A\) and \(\B\) will be elementary equivalent and
minimal.

Before we formally state the reduction let us describe \(\A\) and \(\B\). They
will be models of the theory of the following structure studied by
Shelah~\cite{shelah1978}. The language of the theory contains countably
many unary functions \(F_\nu\) and unary relation symbols \(R_\nu\), one
for each \(\nu\in 2^{<\omega}\). Consider the structure
\[ \mc S=(2^\omega, \langle F_\nu\rangle_{\nu\in 2^{<\omega}}, \langle R_\nu\rangle_{\nu\in 2^{<\omega}})\]
where \(F_\nu\) is defined by
\(F_\nu(\sigma)(x)=\sigma(x)+\nu(x)\mod 2\) where we assume that
\(\nu(x)=0\) for \(x\geq|\nu|\) and \(R_\nu(\sigma)\) if and only if
\(\nu\prec \sigma\). 
Shelah showed that the theory of $\mc
S$ has quantifier elimination and that each element of $\mc S$ generates an
elementary substructure that is minimal.

Let $\hat{\mc{S}_0}$ be the substructure of
$\mc S$ generated by \(\bar 0\), the constant string of \(0\)'s and
$\hat{\mc{S}_1}$ be the substructure generated by \(\bar 1\), the constant
string of \(1\)'s.
These structures are countable and by Shelah's argument, 
\(\hat{\mc{S}}_0\equiv \hat{\mc S}_1\equiv \mc S\).  Furthermore, \(\hat{\mc S}_0\) and \(\hat{\mc S}_1 \) are
minimal models of $Th(\mc S)$. To see this, let \(x\in \hat S_0\), then \(x=F_\nu(\bar 0)\) and
in particular, \(\bar 0=F_\nu(x)\) for some $\nu\in 2^{<\omega}$. So, the substructure of \(\hat{\mc{S}}_0\)
generated by \(x\) is already \(\hat{\mc{S}}_0\).

As we require our structures in \(\mf C\) to be of relational syntax we
will let \(\mc S_0\) and \(\mc S_1\) be the structures corresponding to
\(\hat{\mc S}_0\), respectively \(\hat{\mc S}_1\), after we replace each
\(F_\nu^{\mc S_i}\)
by its graph \(graph_{F_\nu}^{\mc S_i}=\{(\sigma,F_\nu^{\mc
S_i}(\sigma)):\sigma\in S_i\}\). We may assume without loss of generality
that the universes of \(\mc S_0\) and \(\mc S_1\) are \(\omega\) and let
$\A=\mc S_0$ and $\B=\mc S_1$.

Let us describe the structures in the class \(\mf C\) more formally. The
class of structures \(\mf C\) consists of all countable structures with
universe \(\omega\) in the language consisting of a unary relation
\(W\), binary relations \(R_\nu\) and \(graph_{F_\nu}\) for all
\(\nu\in 2^{<\omega}\), and a ternary relation \(O\). We are now ready
to give the function \(f: \mf G \ra \mf C\) witnessing the reduction.

We formally describe how to obtain a structure in \(\mf C\) given a graph.
Let \(\mc G\) be a graph and partition \(\omega\) into countably many
infinite, coinfinite subsets \((A_i)_{i\in\omega}\). Then

\begin{itemize}
\tightlist
\item
  for every \(a_i\in A_0\), \(W^{f(\mc G)}(a_i)\) (we will call elements
  of \(A_0\) the vertices of \(f(\mc G)\)),
\item
  for every \(m,n\in\omega\), if \(mEn\), then for all $\nu\in 2^{<\omega}$ define
  \(R_\nu^{f(\mc G)}\) and \({graph_{F_\nu}}^{f(\mc G)}\) on
  \(A_{\langle m,n\rangle+1}\) such that
  \((A_{\langle m,n\rangle+1},\langle{graph_{F_\nu}}^{f(\mc G)}\rangle_{\nu\in
  2^{<\omega}},R_\nu^{f(\mc G)})\cong \mc S_0\),
\item
  for every \(m,n\in\omega\), if \(\neg mEn\), then for all $\nu\in
  2^{<\omega}$ define 
  \(R_\nu^{f(\mc G)}\) and \({graph_{F_\nu}}^{f(\mc G)}\) on
  \(A_{\langle m,n\rangle+1}\) such that
  \((A_{\langle m,n\rangle+1},\langle{graph_{F_\nu}}^{f(\mc G)}\rangle_{\nu\in
  2^{<\omega}},R_\nu^{f(\mc G)})\cong \mc S_1 \),
\item
  for every \(m,n\in\omega\), let \(O^{f(\mc G)}(a_m,a_n,j)\) for
  all \(j\in A_{\langle m,n\rangle +1}\).
\end{itemize}
This finishes the construction of \(f(\mc G)\). We will refer to the
substructure on the elements in $A_{\langle m,n\rangle+1}$ as the
substructure associated to the pair $(a_m,a_n)$ and to \((A_{\langle m,n\rangle+1},\langle{graph_{F_\nu}}^{f(\mc G)}\rangle_{\nu\in
2^{<\omega}},R_\nu^{f(\mc G)})\) as $\mc S_{(a_m,a_n)}$. 

It is easy to see that
the function \(f\) so defined is Borel, indeed it is even computable.
To see that \(f\)
is a reduction from \(\emb_\mf G\) to \(\eemb_\mf C\) it remains to
prove the following.
\begin{lemma}\label{lem:redupreserveemb}
For \(\mc G,\mc H\in \mf G\), \(\mc G\emb \mc H\) if and only if
\(f(\mc G)\eemb f(\mc H)\).
\end{lemma}

\begin{proof}
That \(\mc G\emb \mc H\) if \(f(\mc G)\eemb f(\mc H)\) follows trivially
from the construction. To show the converse we will use the following model
theoretic fact: For two $\mc L$-structures $\A$ and $\B$, $\A$ is an elementary substructure of $\B$ if and
only if for every finite $\mc R\subseteq \mc L$, the $\mc R$ reduct of $\A$ is
an elementary substructure of the $\mc R$ reduct of $\B$.
Necessity follows trivially from the fact that $\mc R\subseteq \mc L$ and
sufficiency is easily seen by noticing that every first order formula $\phi$ is
in a finite $\mc R_\phi\subseteq \mc L$.

So, say \(\mc G\emb \mc H\) by \(h\). We get an
induced embedding \(\hat h\) defined such that for all $i,i'\in G$, if \(h(i)=j\),
then \(\hat h(a_i)= a_j\) and $\hat h$ is the canonic isomorphism between the
substructure associated to $(a_i, a_{i'})$ and the one associated to $(\hat
h(a_i),\hat h(a_{i'}))$. Without loss of generality we may assume that $f(\mc
G)$ is a substructure of $f(\mc H)$, i.e., that $\hat h$ is the identity.
We use Ehrenfeucht-Fraïssé games to verify that in every finite $\mc R\subseteq \mc
L$, $f(\mc G)$ is an elementary substructure of $f(\mc H)$. We assume without
loss of generality that $\mc R=\{O,W,R_{\nu_0},\dots, R_{\nu_k},
Graph_{F_{\nu_0}},\dots, Graph_{F_{\nu_k}}\}$ where $\nu_i$ is the $i^{th}$
string in
the lexicographical ordering of $2^{<\omega}$ and $k\in\omega$. Let us show
that player \II{} has a winning strategy in $G_m((f(\mc G),g_1,\dots,g_n),(f(\mc
H),g_1,\dots,g_n))$ for arbitrary $m\in\omega$ played in $\mc R$. First, notice that since $\mc S_0\equiv \mc S_1$, \II{}
has a winning strategy for $G_m(\S_0, \S_1)$ in the reduct
$\{R_{\nu_0},\dots, R_{\nu_k}, Graph_{F_{\nu_0}},\dots,Graph_{F_{\nu_k}}\}$. The
following is a winning strategy for $G_m((f(\mc G),g_1,\dots,g_n),(f(\mc
H),g_1,\dots,g_n))$ played in $\mc R$. Say that at turn $i$, the played substructures are $G_i$
and $H_i$ given by the partial isomorphism $h_i$. Assume we are on turn $i+1$.
\begin{enumerate}
  \item If \I{} plays an element $c$ in the $O$-closure of $g_1,\dots,g_n$,
    then let $h_{i+1}(c)=c$.
  \item If \I{} plays an element $c$ in $f(\mc G)$ not in the $O$-closure of $G_i$,
    say it is associated to $(a,b)$ where none of $a,b$ is in the
    $O$-closure of $G_i$, then pick vertices $(a',b')$ in $f(\mc H)$. If $c=a$
    or $b$, let $h_{i+1}(c)=a'$, respectively, $h_{i+1}(c)=b'$. Otherwise start
    running a $G_m(\mc S_{(a,b)},\mc S_{(a',b')})$ winning strategy
    $w_{(a,b)}^{(a',b')}$ and
    let $h_{i+1}(c)=w_{(a,b)}^{(a',b')}(c)$.
  \item If \I{} plays an element $c$ in $f(\mc G)$ not in the $O$-closure of
    $G_i$ but associated to $(a,b)$ where either $a$ or $b$ is in $G_i$, then pick $(a',b')$
    such that $a'$, respectively $b'$, is the element corresponding to $a$,
    respectively $b$, in $H_i$ and continue as
    in (2), mutatis mutandis.
  \item If \I{} plays an element in $f(\mc H)$ not in the $O$-closure of $H_i$,
    then as $f(\mc G)$ is infinite, \II{} can play as in the cases (2) and (3),
    mutatis mutandis.
  \item If \I{} plays an element $c$ in $f(\mc G)$ that is in the
    $O$-closure of $G_i$ but not in the $O$-closure of $g_1,\dots,g_n$, then it
    is associated to some $(a,b)$ in $f(\mc G)$ and by induction there is
    a winning strategy $w_{(a,b)}^{(a',b')}$ that has already been used.
    If $c=a$ or $c=b$, let $h_{i+1}(c)=a'$, respectively, $h_{i+1}(c)=b'$.
    Otherwise let $h_{i+1}(c)=w_{(a,b)}^{(a',b')}(c_1,\dots,c_k,c)$ where
    $c_1,\dots,c_k$ are the elements from the structures associated with
    $(a,b)$ and $(a',b')$ played by $\I{}$ so far.
  \item If \I{} plays an element $c$ in $f(\mc H)$ that is in the
    $O$-closure of $H_i$ but not in the $O$-closure of $g_1,\dots,g_n$, then
    play as in (5), mutatis mutandis.
\end{enumerate}
Since at each turn we play according to winning strategies for games of the
form $G_m(\mc S_i,\mc S_j)$ where $i,j\in \{0,1\}$ we obtain that $h_m$ is a partial isomorphism between $(f(\mc
G),g_1,\dots, g_n)$ and $(f(\mc H), g_1,\dots,g_n)$. We thus have given
a winning strategy for $G_m((f(\mc G),g_1,\dots,g_n),(f(\mc
H),g_1,\dots,g_n))$.
\end{proof}

\subsection{Graphs are complete for elementary
embeddability}\label{sec:graphscomplete}

We will show that for every class of structures \(\mathfrak K\), there
is a computable reduction \(\eemb_\mathfrak K\ra \eemb_\mf G\).

The result we are going to prove appeared in \cite{rossegger2018}.
There, a proof sketch of the fact that the reduction preserves elementary
bi-embeddability spectra was given. We will give a full proof of this fact
in~\cref{sec:degreespectra}. Note that
the coding used in the reduction is not new but was already used in
\cite{andrews2015} to show that graphs are universal for theory spectra.
Let us first describe this coding.

We may assume without loss of generality that \(\mf K\) is a class of
structures in relational language $\mc L=(R_1,\dots)$ where each \(R_i\)
has arity \(i\). Given \(\A\in \mf K\), the graph \(g(\A)\) has  
three vertices \(a\), \(b\), \(c\) where to \(a\) we connect the unique
\(3\)-cycle in the graph, to \(b\) the unique \(5\)-cycle, and to \(c\)
the unique \(7\)-cycle. For each element \(x\in A\) we add a vertex
\(v_x\) and an edge \(a\edge v_x\). For every \(i\) tuple
\(x_1,\dots, x_i\in A\) we add chains of length \(i+k\) for every \(k\),
\(1\leq k \leq i\) with common last elements \(y\). We add an edge
\(v_{x_k}\edge y_1\) only if \(y_1\) is the first element of the chain
of length of \(i+k\). If \(\A\models R_i(x_1,\dots,x_i)\) we add an edge
\(y\edge b\) and otherwise add an edge \(y\edge c\). This finishes the
construction. See \cref{fig:redugraphs} for an example.

\begin{figure}
  \begin{tikzpicture}[->,]
  \tikzset{vertex/.style={circle,fill=black,minimum
  size=3pt,inner sep=2pt}}
    \tikzset{>=latex}
    \node[vertex] (a) at (0,0){};
    \node[vertex] (b) at (6,0.5){};
    \node[vertex] (c) at (6,-0.5){};
    \draw (a) to [out=140,in=220,looseness=20] node[fill=white,inner sep=1pt] {$3$}  (a);
    \draw (b) to [out=320,in=40,looseness=20] node[fill=white,inner sep=1pt] {$5$}  (b);
    \draw (c) to [out=320,in=40,looseness=20] node[fill=white,inner sep=1pt] {$7$}  (c);
    \node[vertex,label={$v_3$}] (v_3) at (1,2){};
    \node[vertex,label={$v_2$}] (v_2) at (1,0){};
    \node[vertex,label=below:{$v_1$}] (v_1) at (1,-2){};
    
    \foreach \x in {1,...,3}
    {\draw (a) -- (v_\x);}
      \node[vertex] (y) at (5,1) {};  
    \foreach \x in {3,...,5}
    {
      \foreach \y in {1,...,\x}
      {
        \pgfmathsetmacro{\ycord}{(\x-3)/2+0.5}
        \node[vertex] (y_\x\y) at (2+\y/2,\ycord){};
      }
      \pgfmathsetmacro{\li}{\x-1}
      \foreach \y in {1,...,\li}
      {
        \pgfmathsetmacro{\h}{\y+1}
        \draw (y_\x\y) -- (y_\x\h);
      }
      \draw (y_\x\x) -- (y);
      \pgfmathsetmacro{\v}{\x-2}
      \draw (v_\v) -- (y_\x1);
    }
    \draw (y) -- (b);
    \node[vertex] (z) at (5,-1) {};  
    
    \foreach \x in {3,...,5}
    {
      \foreach \y in {1,...,\x}
      {
        \pgfmathsetmacro{\ycord}{-(\x-3)/2-0.5}
        \node[vertex] (z_\x\y) at (2+\y/2,\ycord){};
      }
      \pgfmathsetmacro{\li}{\x-1}
      \foreach \y in {1,...,\li}
      {
        \pgfmathsetmacro{\h}{\y+1}
        \draw (z_\x\y) -- (z_\x\h);
      }
      \draw (z_\x\x) -- (z);
    }
    \draw (v_1) -- (z_51);
    \draw (v_2) -- (z_41);
    \draw (v_3) -- (z_31);
    \draw (z) -- (c);
\end{tikzpicture} 
\caption{\label{fig:redugraphs} Part of the graph $F(\A)$ coding that
$\A\not\models R_3(3,2,1)$ and $\A\models R_3(1,2,3)$.}
\end{figure}
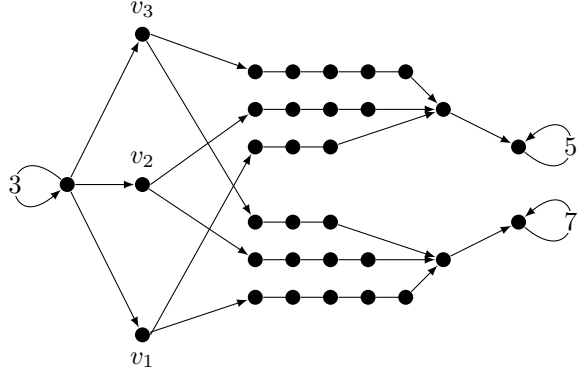 


Let us fix the some notation for the following proofs. Given a structure
$\A$ and $\bar a\in A^{<\omega}$ we let $\langle \bar a\rangle^\A$ be the
substructure of $\A$ generated by $\bar a$.
\begin{lemma}\label{lem:redupreserveeemb}
  For $\A,\B\in \mf K$, $\A\eemb \B$ if and only if $g(\A)\eemb g(\B)$.
\end{lemma}
\begin{proof} 

$(\Ra)$. Assume that \(\mc A \eemb \mc B\) and that
\(\mc A\) is an elementary substructure of $\mc B$. We may also assume without loss of
generality that \(g(\A)\subseteq g(\B)\). We will show that for
all \(n\in \omega\) and any \(\ol a\in g(\mc A)^{<\omega}\) player II
has a winning strategy for the \(n\) turn
Ehrenfeucht Fra\"iss\'e game
\(G_n((g({\mc A}),\ol a),(g({\mc B}),\ol a))\). Assume that \(n\)
is the least such that player II has no winning strategy for
\(G_n((g({\A}),\ol a),(g(\B),\ol a))\). Consider the set of partial
isomorphisms from $(g(\A),\ol a)$ to $(g(\B),\ol a)$. This set can not have the
back-and-forth property. In particular, the back-and-forth property fails
already if we only consider partial isomorphisms with domain of size $n+|\ol a|$.
Otherwise there would be a winning strategy for $G_n((g(\A),\ol a), (g(\B),\ol
a))$. So, either there is
\(\ol v \in g(\A)^{n}\) such that for all \(\ol u \in g(\B)^{n}\),
\(\langle \ol a \ol v\rangle^{g({\A})}\not \cong \langle \ol a  \ol
u \rangle^{g({\B})}\)
or there is \(\ol u\in g(\B)^{n}\) such that for all
\(\ol v \in g(\B)^{n}\),
\(\langle \ol a \ol u \rangle^{g(\B)}\not \cong \langle \ol a\ol
v \rangle^{g({\A})}\).
We will derive a contradiction assuming the second case. Deriving one from the
first case can be done in a similar fashion.

Notice that \(\ol a\ol u\) is in a substructure of \(g({\B})\) coding
a finite substructure of \(\B\) in a finite part \(\mc L_1\) of the
language of \(\B\). Extend \(\langle \ol a \ol u\rangle^{g({\B})}\)
so that it codes such a substructure \(\B_1\) of \(\B\). Consider
the conjunction \(\phi\) of atomic formulas, or negations thereof, true of
\(\mc B_1\) in \(\mc L_1\). Let \(\ol a'\) be the elements in
\(B_1\cap A\) and \(\ol u'\) the elements in \(B_1\setminus A\). Then
\(\mc B \models \phi(\ol a'\ol u')\) and the Tarski-Vaught test gives us
elements \(\ol v'\) in \(\mc A\) such that
\(\mc A\models \phi(\ol a'\ol v')\). It follows that we have a partial
isomorphism between \(\langle \ol a'\ol u'\rangle^{\B}\) and
\(\langle\ol a'\ol v'\rangle^{\A}\) in \(\mc L_1\). This induces an
isomorphism between the subgraph coding \(\B_1\) and the subgraph coding
\(\langle \ol a'\ol v'\rangle^{\A}\). But
\(\langle \ol a \ol u \rangle^{g(\B)}\) is a subgraph of the
graph coding \(\mc B_1\) and thus it is isomorphic to a substructure
\(\langle \ol a\ol v\rangle^{g(\A)}\) of the structure coding
\(\langle \ol a'\ol v'\rangle^{\A}\), a contradiction.

$(\La)$. An easy induction on the quantifier depth of formulas in $\mc L$ shows that for every $\A\in\mathfrak K$ and $\mc L$-formula $\phi$ with $n$-free
variables the set 
\[D_\phi^\A=\{ (v_{a_1},\dots, v_{a_n}) : (\A,a_1,\dots,a_n) \models \phi(a_1,\dots,a_n)\}\]
is definable in $g(\A)$. Now, assume that $g(\A)\eemb g(\B)$ and without loss
of generality that $g(\A)$ is an elementary substructure of $g(\B)$. Let $g_\B:
\B\ra g(\B)$ be defined by $g_\B: b\mapsto v_b$. Notice
that the map $a\mapsto g_\B^{-1}(v_a)$ is an embedding of $\A$ in $\B$. To see
that this embedding is elementary assume that $(\A,\ol a)\models \phi$, then
$\bar v_{\bar a}\in D_\phi^\A$ and by elementarity $\bar v_{\bar a}\in D_\phi^\B$.
So, $(\B, g_\B^{-1}(\bar v_{\bar a}))\models \phi(g_\B^{-1}(\bar v_{\bar a}))$.

\end{proof}

Concatenating the reductions $f$ and $g$ and from the fact that $\emb_\mf G$,
$\biemb_\mf G$ are complete $\pmb \Sigma^1_1$ quasi-orders, respectively
equivalence relations, we obtain \cref{thm:elbeborel}. 

To prove \cref{thm:elbecomp} notice that $f$ and $g$ are computable. Thus there
is a Turing operator $\Phi$ such that $\Phi=g\circ f$. We can find a Turing machine $\phi_i$
such that $\phi_i(j,k)=\Phi^{\A_j}(k)$ for all $k\in\omega$ if $\A_j$ is
a total computable structure. Using the s-m-n theorem we can then get a computable
function $j\mapsto u(i,j)$ where $u(i,j)$ is an index for $\Phi^{\A_j}$. Thus
$\emb_\mf G$ is computably reducible to $\eemb_\mf G$ as a quasi-order on
$\omega$. Fokina and Friedman~\cite{fokina2009} showed that $\emb_\mf G$ is
$\Sigma_1^1$ complete. Thus, $\eemb_\mf G$ is also $\Sigma^1_1$ complete and
\cref{thm:elbecomp} follows.
\section{Degree spectra}\label{sec:degreespectra}
In this section we finish the proof of \cref{thm:dgsp}. As noticed before the
two reductions $f: \mf G\ra \mf C$ and $g:\mf C\ra \mf G$ are computable. We
will see that the two functions induce an even stronger notion of reduction
that allows us to relate the degree spectra realized by $\biemb_\mf G$ and
$\elbe_\mf G$.
\begin{definition}[cf.~\cite{harrison-trainor2017,miller2018}]
    Let $\mf C$ and $\mf D$ be categories. A \emph{computable functor} between $\mf C$ and $\mf D$ is a pair of computable operators $(\Phi,\Phi_*)$ such that
    \begin{enumerate}
      \item for all $\A\in \mf C_1$, $F(\A)=\Phi^{\A}$,
      \item for all $f:\A\ra \B\in \mf C_2$, $F(f)=\Phi_*^{\A\oplus f\oplus \B}$.
    \end{enumerate}
\end{definition}
Computable functors preserve many computability theoretic properties. One
example are degree spectra: Recall
that for $X,Y\subseteq \mathcal P(\omega)$, $X$ is \emph{Medwedev reducible} to
$Y$, $X\leq_s Y$, if there is a Turing operator $\Phi$ such that for all $y\in
Y$, there is $x\in X$ such that $\Phi^y=x$. We have in
particular that if $F:(\mf C,\ar_1)\ra (\mf D, \ar_2)$ is a computable functor, and
$\sim_i$ is the equivalence relation given by \[ \A\sim_i\B \LR \A\ar_i \B
\land \B \ar_i \A,\] then for all $\A\in \mf C$, 
$DgSp_{\sim_1}(F(\A))\leq_s DgSp_{\sim_2}(\A).$

It is an easy exercise to see that $g\circ f$ induces a computable functor $H:(\mf
G,\emb)\ra (\mf G,\eemb)$ and thus for all $\mc G\in \mf G$,
$DgSp_\biemb(H(\mc G))\leq_s DgSp_\elbe(\mc G)$.

To get that every degree spectrum realized in $\mf C$ is also realized in
$\mf D$ we need a stronger notion of reducibility. To define this we need an
effectivization of the category theoretic notion of a natural isomorphism
between functors.
\begin{definition}[\cite{harrison-trainor2017}]\label{def:effiso}
  A functor $F:\mf C\ra \mf D$ is \emph{effectively isomorphic} to $G:\mf C\ra
  \mf D$ if there is a Turing operator $\Lambda$ such that for every
  $\A\in \mf C$, $\Lambda^\A$ is an isomorphism from $F(\A)$ to $G(\A)$, and the following diagram
  commutes for all $\A,\B\in \mf C_1$ and every $\gamma:\A\ra \B\in \mf C_2$.\\
  \begin{center}
  \begin{tikzcd}
    F(\A)\ar[r,"\Lambda^\A"]\ar[d,"F(\gamma)"] & G(\A)\ar[d,"G(\gamma)"]\\
    F(\B)\ar[r,"\Lambda^\B"] & G(\B)
  \end{tikzcd}
\end{center}
\end{definition}
\begin{definition}[cf.~\cite{harrison-trainor2017}]\label{def:cbf}
  We say that $(\mf C,\ar_1)$ is \emph{CBF-reducible} to $(\mf D,\ar_2)$, $(\mf
  C,\ar_1)\leq_{CBF}(\mf D,\ar_2)$ if 
  \begin{enumerate}
    \item there is a computable functor $F:\mf C \ra \mf D$ and a computable functor $G:\mf
  D \supseteq \hat{\mf D}\ra \mf C$ where $\hat{\mf D}$ is the $\sim_2$-closure 
  of $F(\mf C )$,
    \item $F\circ G$ is effectively isomorphic to $Id_{\hat{\mf D}}$, $G\circ
      F$ is effectively isomorphic to $Id_\mf C$,
    \item and, if $\Lambda_\mf C$, $\Lambda_\mf D$ are the operators witnessing
      the effective isomorphism between $G\circ F$ and $Id_\mf C$,
      respectively, $F\circ G$ and $Id_{\hat{\mf D}}$, then for every $\A\in \mf C$,
      $F(\Lambda_\mf C^\A)=\Lambda_\mf D^{F(\A)}:F(\A) \ra F(G(F(\A)))$ and every
      $\B\in \hat{\mf D}$, $G(\Lambda_\mf D^\B)=\Lambda_\mf C^{G(\B)}:G(\B)\ra
      G(F(G(\B)))$.
  \end{enumerate}
\end{definition}
Consider two structures $\A$ and $\B$ and a morphism $f:\A\cong\B$. Then,
clearly $\A\leq_T \B\oplus f$; after all, we have that $R^\A(a_1,\dots,
a_n)$ if and only if $R^\B(f(a_1),\dots, f(a_n))$. The following definition
generalizes this observation.
\begin{definition}
  A category $\mf C$ is \emph{degree invariant} if for every
  $\A,\B\in\mf C_1$ and every $f:\A\ra\B\in \mf C_2$, $f\equiv_T f^{-1}$ and $\A\leq_T
  \B\oplus f$.
\end{definition}
\begin{proposition}\label{prop:specpreserving}
  If $\mf C$ and $\mf D$ are degree invariant and $\mf C\leq_{CBF} \mf D$, then every set realized as a $\sim_1$-spectrum
  in $\mf C$ is realized as a $\sim_2$-spectrum in $\mf D$.
\end{proposition}
\begin{proof}
  Say $F:\mf C \ra \mf D$ and $G:\mf D\ra \mf C$ witness that $\mf
  C\leq_{CBF}\mf D$. 
  Fix $\A\in \mf C$ and let $\Lambda$ be the Turing operator witnessing that
  $G\circ F$ is effectively isomorphic to the identity functor on $\mf C$.
  Then, for $\hat\A\sim_1\A$, $\hat \A\geq_T F(\hat\A)\geq_T G(F(\hat \A))$ and by degree
  invariance $\hat \A\leq_T \Lambda^\A\oplus G(F(\hat\A))\equiv_T
  G(F(\hat\A))\leq_T F(\hat\A)$. Thus, $DgSp_{\sim_1}(\A)\subseteq
  DgSp_{\sim_2}(F(\hat\A))$. The proof that $DgSp_{\sim_1}(\A)\supseteq
  DgSp_{\sim_2}(F(\hat\A))$ is similar. So, if $X$ is a $\sim_1$ spectrum
  realized by $\A$ in $\mf C$, then it is realized as a $\sim_2$ spectrum in
  $\mf D$.
\end{proof}
Notice that if $\mf K$ is a class of relational structures, then whether $(\mf
K, \ar)$ is degree invariant only depends on $\ar$. Thus we might say that
a relation on structures is degree invariant.
\begin{definition}
  A class of structures $\mf C$ is \emph{CBF-complete} with respect to
  a degree invariant relation $\ar$, if for every class $\mf K$, $(\mf K, \ar)\leq_{CBF} (\mf
  C,\ar)$.
\end{definition}
We showed in \cref{sec:graphscomplete} that for any class $\mf K$ equipped with the elementary embeddability relation there is a computable reduction $g$ from $(\mf K,\eemb)$ to $(\mf
G,\eemb)$. We can now show that these reductions induce $CBF$-reductions $(\mf
K,\eemb)\leq_{CBF} (\mf G\eemb)$ and that thus graphs are CBF-complete for
elementary embeddability. Verifying the conditions of \cref{def:cbf} is quite
technical, but the core ideas of the proof should not be too
difficult.
\begin{theorem}\label{thm:cbfcomplete}
  The class of graphs is CBF-complete for elementary embeddability.
\end{theorem}
\begin{proof}
  Fix a class $\mf K$. It is clear from the construction that $g$ induces
  a computable functor $F:(\mf K,\elbe)\ra (\mf G,\elbe)$. We have to show that $F(\mf K)$
  is closed under elementary bi-embeddability, that there is a functor
  $G:F(\mf K)\ra \mf K$ such that $F\circ G$ and $G\circ F$ are effectively
  isomorphic to the identity on $\mf K$, respectively, $F(\mf K)$ and that the
  witnesses of these effective isomorphisms agree.

  Let $\mc G\elbe F(\A)$ for some $\A\in \mf K$. We may assume without loss of
  generality that $\mc G$ is an elementary substructure of $F(\A)$. For every
  $\bar a\in \mc G^{<\omega}$, $tp_\mc G(\bar a)=tp_\A(\bar a)$. Thus $\mc G$
  must contain the elements $a,b,c$ of $F(\A)$ with unique $3$-cycles, respectively,
  $5$-cycles and $7$-cycles connected to them. Furthermore, say $\bar a \in \mc G$
  codes elements of $A$ in $F(\A)$ such that $\A\models R_i(\bar a)$, then this
  information must also be coded in $\mc G$ as it is preserved in the type of
  $\bar a$. We can compute a structure $G(\mc G)$ as follows. Fix a $\mc G$
  computable injective enumeration $f$ of the set $\{x: a\edge x\}$. Notice that this can
  be done uniformly since the set $\{x:a\edge x\}$ is uniformly computable in
  all structures in $F(\mf K)$. Let the universe of $G(\mc G)$ be the pull-back
  along $f$. Then for all $a_1,\dots, a_i=\bar
  a \in\omega^{i}$,
  $G(\mc G)\models R_i(\bar a)$ if for every $a_j$, $j<i$, there is a chain of
  $i+j$ connected elements $y_1,\dots y_{i+j}$ with $f(a_j)\edge y_1$, all $j$
  chains share the same last element $y$ and $y\edge b$. Likewise, $G(\mc
  G)\models \neg R_i(\bar a)$ if there are chains satisfying the above
  conditions with $y\edge c$. This finishes the construction of $G(\mc G)$. 
  
  Let $\mc G,\hat{\mc G}\in F(\mf K)$ and $g:\mc G\eemb \hat{ \mc G}$. As both
  graphs are elementary bi-embeddable with images of structures in $\mf K$,
  they have unique vertices $a$, respectively, $\hat a $ with $3$-cycles connected to them. Computably in $\mc G$ and $\hat{\mc G}$ find the
  vertices and enumerate the sets $\{ x: a\edge x\}$, and $\{ x: \hat a\edge
  x\}$ using the same procedure as in the construction of $G(\mc G)$ above. Let
  $f$, respectively, $\hat f$ be these enumerations. Now let $G(g)=\hat
  f^{-1}\circ g\circ f$. By construction $G(g):G(\mc{G})\emb G(\hat{\mc
  G})$ and $G(g)$ is uniformly computable in $\mc G\oplus g\oplus \hat{\mc G}$.
  To see that $G(g)$ is elementary, assume towards a contradiction that it is
  not. Then there is $\ol
  a\in G(\mc G)$ and $\phi$ such that $G(\mc G)\models \phi(\ol a)$ but
  $G(\hat{\mc G})\not\models \phi(G(g)(\ol a))$. Recall that the atomic diagram
  of the tuple $\ol a$ is coded in the type of $f(\ol a)$ in $\mc G$ and
  similarly, the atomic diagram of $G(g)(\ol a)$ is coded in the type of
  $g(f(\ol a))$ in $\hat{ \mc G}$. So, $g$ could not be elementary, a contradiction.

  To see that $G\circ F$ and $F\circ G$ are
  effectively isomorphic to the identities on $\mf K$ and $F(\mf K)$,
  respectively, first note that $G(F(\A))\cong \A$. There is a canonic
  isomorphism given by the composition of the maps $a \mapsto v_a$ and the
  enumeration $f$ of the set $\{x: a\edge x\}$, i.e., the isomorphism is
  defined by $a\mapsto f^{-1}(v_a)$. It is clearly uniformly computable, say by
  $\Lambda_\mf K$. On the other hand let $\mc G \in F(\mf K)$, then we can
  compute an isomorphism between $F(G(\mc G))$ and $\mc G$ by doing the
  following. Every $v\in G$ either defines a relation $R_i$ on some tuple, codes an element, or is used to define $a, b, c$. One can computably
  determine which of the three cases holds. In the second case simply map
  $v$ to $v_{f^{-1}(v)}$, in the third case one can computably determine whether
  $v$ is used to define $a, b, c$ and, using $F$ and $G$, computably find the
  corresponding element in $F(G(\mc G))$. In the first case, we have to find
  the tuple $\ol w$ such that $v$ is involved in the coding of the relation
  $R_i$ on $\bar w$. We then map $v$ to the
  corresponding element in the coding of $R_i$ on the tuple $v_{f^{-1}(\ol w)}$.
  It is easy to see that one can define a Turing operator $\Lambda_\mf{F(\mf
    K)}$ computing this isomorphism. The Turing operators $\Lambda_\mf{F(\mf
  K)}$ and $\Lambda_{\mf K}$ will witness the effective isomorphism between
  $F\circ G$ and the identity on $F(\mf K)$, respectively, the effective isomorphism
  between $G\circ F$ and the identity on $\mf G$.
  
  It remains to show that the diagrams of \cref{def:effiso} commute and that
  for all $\A\in \mf K$ and $\mc G\in F(\mf K)$, 
  $F(\Lambda_\mf K^\A)=\Lambda_{F(\mf K)}^{F(\A)}$ and $G(\Lambda_{F(\mf K)}^\mc
  G)=\Lambda_\mf K^{G(\mc G)}$. For the commutation of the diagrams, say first
  that $\A,\hat{\A}\in\mf K$ with $\iota: \A\eemb \hat{ \A}$. Let $h:\A\ra F(\A)$ and
  $\hat h: \hat \A\ra F(\A)$ given by $h,\hat h:a\mapsto v_a$. We have not
  given an explicit definition of $F(\iota)$ yet. But notice that $F(\iota)$ is
  uniquely determined by the way it maps the elements $v_a$. In particular, if
  $\nu(x)=F(\iota)(x)$ on the elements with $a\edge x$, then $\nu=F(\iota)$.
  Thus we have that
\[ G(F(\iota))=\hat f^{-1}\circ F(\iota)\circ f=\hat f^{-1}\circ
\hat h \circ\iota\circ h^{-1}\circ f,\]
and $\Lambda_{\mf K}^\A=f^{-1}\circ h$, so
\[ \Lambda_\mf K^{\hat \A}\circ \iota=\hat f^{-1} \circ\hat h\circ
\iota=G(F(\iota))\circ \Lambda_\mf K^\A\] 
and thus $G\circ F$ is effectively isomorphic to $id_\mf K$. Now, say $\mc
G,\hat{\mc G}\in F(\mf K)$ with $\eta:\mc G\eemb \hat{\mc G}$. First let $x\in
\mc G$ with $a\edge x$. Let $h$, and $\hat h$ be as above, then $(\Lambda_{F(\mf K)}^{\hat{ \mc G}}\circ\eta)(x)=(\hat
h\circ\hat f^{-1}\circ\eta)(x)$ and $F(\hat f^{-1}\circ \eta\circ f)=\hat
h\circ\hat f^{-1}\circ\eta\circ f\circ h^{-1}$, so
\[(F(G(\eta))\circ\Lambda_{F(\mf K)}^\mc
  G)(x)=(F(\hat f^{-1}\circ \eta\circ f)\circ h\circ f^{-1})(x)=(\hat h\circ
\hat f^{-1}\circ \eta)(x).\]
Having established that the diagram commutes on the restricted universes we
use the fact that any embedding is determined by these parts of the universes
to obtain that $F\circ G$ is 
effectively isomorphic to $id_{F(\mf K)}$.

To verify the last condition in \cref{def:cbf} let $\A\in \mf K$, then on
$\{x:x\edge a\}$
\[ F(\Lambda_{\mf K}^\A)(x)=(\hat h\circ f^{-1}\circ h\circ h^{-1})(x)=(\hat h\circ
f^{-1})(x)=\Lambda_{F(\mf K)}^{F(\A)}(x)\]
and as there is a unique extension of this to a mapping $F(\A)\ra F(G(F(\A)))$
$F(\Lambda_{\mf K}^\A)=\Lambda_{F(\mf K)}^{F(\A)}$. At last, let $\mc G\in
F(\mf K)$, then 
\[ G(\Lambda^{\mc G}_{F(\mf K)})=\hat f^{-1}\circ h\circ f^{-1}\circ f=\hat
f^{-1}\circ h=\Lambda_\mf K^{G(\mc G)}\]
where $f$ is the enumeration of $\{x:x\edge a\}$ in $\mc G$ and $\hat f$ the
one in $G(F(\mc G))$.
\end{proof}
The following is a direct consequence of
\cref{prop:specpreserving,thm:cbfcomplete}.
\begin{corollary}\label{cor:graphscompletespectra}
  For every structure $\A$, there is a graph $\mc G_\A$ such that 
  \[ DgSp_\elbe(\A)=DgSp_\elbe(\mc G_\A).\]
\end{corollary}
Unfortunately, for the reduction from bi-embeddability on graphs to elementary
bi-embeddability on $\mf C$ given in \cref{sec:reduemb_geemb_c} we can not
deduce that $(\mf G,\emb)\leq_{CBF} (\mf C,\eemb)$. 
However, we can still establish a relationship between the degree spectra in
these classes. Recall that $\mc S_0$ is the substructure of $\mc S$ generated by the
constant string of $0$'s and $\mc S_1$ is the substructure generated by the
constant string of $1$'s.
\begin{lemma}\label{lem:jumpinv}
Let $X$ be $\Delta^0_2(Y)$ for some set $Y$. Then there exist
a sequence of structures $(\mc C_i)_{i\in\omega}$, uniformly computable in $Y$,
such that for all $i\in\omega$
\[ \mc C_i\cong\begin{cases} \mc S_0 & \text{if }i\in X\\ \S_1 &\text{if } i\not\in X\end{cases}.\]
\end{lemma}
\begin{proof}
  As $X$ is $\Delta^0_2$ there is an $X$-computable two valued function $f$ such that 
  \[ \lim_{s\ra \infty} f(i,s)=\begin{cases} 0 &\text{if } i\in X\\ 1&\text{if } i\not\in
  X\end{cases}.\]
  Define a structure $\C$ as follows. Fix an enumeration $g$ of $2^{<\omega}$. At stage $0$
  define $\C_0$ to be the partial structure containing one element $a$ on which
  no relation holds and leave all function symbols undefined. 
  Say we have defined the structure $\C_s$. At stage $s+1$ we look at $f(i,j)$
  for $j<s$ and define $\C_{s+1}$ as if $a$ was the finite string with
  $a(j)=f(j)$ for $j<s$. To be more precise:
  \begin{enumerate}
    \item For all $k$, if $k\leq s$ and $|g(k)|\leq s$ then let $R_{g(k)}(a)$
      if and only if $g(k)\preceq a$, and if $F_{g(k)}(a)$ has not been defined yet add
      a new element and set $F_{g(k)}(a)$.
    \item We may assume by induction that for all elements $b$ in $\C_{s+1}$
      there is $k\leq s$ such that $b=F_{g(k)}(a)$. We set $R_{g(l)}(b)$
      respecting this equation for all $l\leq s$.
  \end{enumerate}
  It is easy to see that this procedure yields a computable sequence of
  structures $\C_s$ with $\C_s\subseteq \C_{s+1}$ and a computable structure as
  its limit. We let $\C$ be this structure. $\C$ contains an element $a$ such
  that $\A\models R_\sigma(a)$ if and only if $a\preceq f(i,-)$ and all other
  elements are equal to $F_\tau(a)$ for some $\tau\in 2^{<\omega}$. Thus, in
  particular if $\lim f(i,s)=0$, then there is an element representing the
  constant string of $0$'s in $\A$ and otherwise there is an element representing
  the constant string of $1$'s in $\A$. Let $\C_i=\C$, then $\C_i\cong \S_0$ if and only if
  $i\in X$ and $\C_i\cong \S_1$ if and only if $i\not\in X$ as required.
\end{proof}
We use the usual category theoretic definition of \emph{pseudo-inverse}. Two
functors $F:\mf C\ra \mf D$ and $G:\mf D\ra\mf C$ are pseudo-inverses if
$F\circ G$ is naturally isomorphic to $id_{\mf D}$ and $G\circ F$ is naturally
isomorphic to $id_{\mf C}$. 

Recall that a structure $\A$ is automorphically trivial if there is
a finite set $D\subseteq A $ such that every permutation of $A$ that fixes $D$
pointwise is an automorphism. Knight~\cite{knight1986} showed that isomorphism
spectra of automorphically trivial structures contain only one Turing degree
and that the isomorphism spectra of automorphically non-trivial structures are upwards closed in the Turing degrees.
In~\cite{fokina2019a} the authors showed that if $\A$ is automorphically
trivial and $\B\biemb \A$, then $\B\cong \A$. Thus, as every bi-embeddability
and elementary bi-embeddability spectrum is a union of isomorphism spectra,
Knight's result carries over to this setting.
\begin{lemma}\label{lem:relationshipspectra}
  For every automorphically non-trivial structure $\mc G\in \mf G$ there is $\mc A\in \mf C$ such that
  \[DgSp_\elbe(\mc A)=\{X: X'\in DgSp_\biemb(\mc G)\}.\]
\end{lemma}
\begin{proof}
  Recall the reduction from embeddability on graphs to elementary embeddability
  on $\mf C$ given in \cref{sec:reduemb_geemb_c}. It is easy to see that it
  induces a computable functor $F: (\mf G, \biemb)\ra(\mf C,\elbe)$. We
  show that the functor has a pseudo-inverse $G$ on the
  $\elbe$-saturation of $F(\mf G)$ and then use \cref{lem:jumpinv} to obtain
  the lemma. The minimality of the submodels $\mc S$ used in the construction
  of $F$ will play a crucial role in the proof.

  Say $\B\elbe F(\mc G)$ for $\mc G\in \mf G$, that $x,y$ are vertices in $\B$
  and $\mc S_{(x,y)}$ is the substructure on the elements satisfying $O(x,y,-)$
  in the reduct to the language of $\mc S$.
  We have that either $\mc S_{(x,y)}\cong \mc S_0$ or $\mc S_{(x,y)}\cong \mc
  S_1$ since it elementary embeds into $\mc S_{(u,v)}$ for
  some $u,v\in F(\mc G)$ and $\mc S_{(u,v)}\cong \mc S_0$ or $\mc S_{(u,v)}\cong
  \mc S_1$ by minimality. Thus, we get a graph $G(\B)$ from $\B$ by defining an edge between
  two
  vertex variables $x,y$ from $\B$ if and only if $\mc S_{(u,v)}\cong \mc S_0$. Clearly
  every elementary embedding of $\B$ in $F(\mc G)$ yields an embedding of
  $G(\B)$ in $\mc G\cong G(F(\mc G))$ and the analogous fact is true for every elementary
  embedding of $F(\mc G)$ in $\B$. Likewise, we can argue that $F(G(\A))\cong
  \A$ for every $\A\in F(\mf G)$. Thus $G$ and $F$ are pseudo-inverses.
  
  However, notice that $G$ is not effective. Within one jump over the diagram
  of any $\B\in F(\mf G)$ we can compute $G(\mc
  B)$ as the isomorphism types of $\mc S_1$ and $\mc S_0$ are definable by
  $\Sigma^c_2$ formulas in $\mf C$. This implies that for all
  $\A\in \mf G$, $F(\A)'\geq_T G(F(\A))\cong \A$. So, in particular, 
  \begin{equation}\label{eq1}
    DgSp_\elbe(F(\A))\geq_s
  \{X: X'\in DgSp_\biemb(G(F(\mc A)))=DgSp_\biemb(\A)\}.\end{equation}
On the other hand, let $X\in DgSp_\biemb(\A)$ and $\hat \A\biemb \A$ such that
$\hat\A\equiv_T X$. Then by \cref{lem:jumpinv} for every $Y$ with $Y'\geq_T X$,
there is $\B\cong F(\hat\A)$ with $\B\equiv_T Y$. This process is uniform in
$\B$ and $Y$. Thus
\begin{equation}\label{eq2} DgSp_\elbe(F(\A))\leq_s \{ X: X'\in
DgSp_\biemb(\A)\}.\end{equation}
As both bi-embeddability and elementary bi-embeddability spectra of
automorphically non-trivial structures are upwards closed, \cref{eq1,eq2} imply
that
\[ DgSp_\elbe(F(\A))=\{ X: X'\in DgSp_\biemb(\A)\}.\]
\end{proof}
\cref{thm:dgsp} follows directly from
\cref{lem:relationshipspectra} and \cref{cor:graphscompletespectra}.

We note that \cref{thm:dgsp} may not be optimal. Using a different proof one might
be able to get an even stronger relationship between the spectra realized by
bi-embeddability on graphs and elementary bi-embeddability on graphs. We thus
ask.
\begin{question}
  Is every bi-embeddability spectrum of a graph the elementary bi-embeddability
  spectrum of a graph and vice versa?
\end{question}
One way to answer this question positively is by showing that if $X$ is an
elementary bi-embeddability spectrum then so is $X'=\{ x': x\in X\}$. This is true for
isomorphism spectra and usually shown by considering an appropriate definition
for the jump of a structure. However, all known definitions do not
preserve elementary embeddability (and not even elementary equivalence). We
  thus ask.
\begin{question}
  Let $X$ be the elementary bi-embeddability spectrum of a graph. Is $X'$ the
  elementary bi-embeddability spectrum of a graph?
\end{question}
\begin{question}
  Let $X$ be the theory spectrum of a graph. Is $X'$ the theory spectrum of
  a graph?
\end{question}

Also, while \cref{thm:cbfcomplete} shows that graphs are complete for
elementary bi-embeddability spectra, it is unknown whether the same is true for
bi-embeddability.
\begin{question}
  Is every bi-embeddability spectrum of a structure realized as the
  bi-embeddability spectrum of a graph?
\end{question}
\printbibliography
\end{document}